\documentclass[11pt]{amsart}
\usepackage{amsfonts,latexsym,rawfonts,amsmath,amssymb,amsthm}
\usepackage[plainpages=false]{hyperref}
\usepackage{graphicx}

\RequirePackage{color}

\theoremstyle{plain}
  \newtheorem{theorem}{Theorem}[section]
  \newtheorem{lemma}{Lemma}[section]

  \newtheorem{definition}{Definition}[section]

\theoremstyle{remark} 
  \newtheorem{remark}{Remark}[section]

\renewcommand{\det}{\mbox{det}}
  
  \newcommand{\dist}{\mbox{dist}}
  \newcommand{\diam}{{\mbox{diam}}}

  \numberwithin{equation}{section}
  \numberwithin{figure}{section}

\begin{document}

\title[Light reflection is nonlinear optimization]
{Light reflection is nonlinear optimization}


\author[Jiakun Liu]
{Jiakun Liu}

\address
	{Department of Mathematics,
	Princeton University,
	Fine Hall, Washington Road,
	Princeton, NJ 08544-1000, USA}
	
\email{jiakunl@math.princeton.edu}

\thanks{The author is supported by the Simons Foundation}
\thanks{\copyright 2011 by the author. All rights reserved}

\subjclass[2000]{35J60, 78A05;\, 90C30, 49N45}

\keywords{Nonlinear optimization, Monge-Amp\`ere equation}

\begin{abstract}
In this paper, we show that the near field reflector problem is a nonlinear optimization problem.
From the corresponding functional and constraint function, we derive the Monge-Amp\`ere type equation for such a problem.  
\end{abstract}

\maketitle

\baselineskip=16.4pt
\parskip=3pt

\section{Introduction}

Optimal transportation, due to its various applications, has been extensively studied in recent years. 
The modern theory of optimal transportation is mainly built upon Kantorovich's dual functional, which is a linear functional subject to a linear constraint. 
With his dual functional, Kantorovich introduced linear programming, which is a class of linear optimization problems. 
An important new application is the reflector design problem. 
In \cite{W04}, Xu-Jia Wang showed that the far field case of the reflector design problem is an optimal transportation problem, and so is a linear optimization problem. 
The purpose of this paper is to show that the general case of the reflector problem is a nonlinear optimization problem. 
More examples of nonlinear optimization problems and also questions of the existence and regularity of potential functions and optimal mappings will be investigated in \cite{Liu} and subsequent papers.

Suppose that a point source of light is centered at the origin $O$ and for each $X\in\Omega\subset\mathbb{S}^n$ we issue a ray from $O$ passing through $X$, which after reflection by a surface $\Gamma$ will illuminate a point $Y$ on the target surface $\Omega^*$ in $\mathbb{R}^{n+1}$. 
Let $f\in L^1(\Omega), g\in L^1(\Omega^*)$ be the input and gain densities, and $d\mu, d\nu$ denote the surface area elements of $\Omega, \Omega^*$, respectively. The near field reflector problem can be formulated as follows: given $(\Omega, f)$ and $(\Omega^*, g)$ satisfying the energy conservation condition
	\begin{equation}\label{e101}
		\int_\Omega fd\mu=\int_{\Omega^*} gd\nu,
	\end{equation}
find a reflector $\Gamma$ such that the light emitting from $\Omega$ with density $f$ is reflected off $\Gamma$ to the target $\Omega^*$ and the density of reflected light is equal to $g$.

In our reflector problem, we assume that both $\Omega$ and $\Omega^*$ are compact and each has boundary of measure zero.
Represent the reflector $\Gamma$ as a radial graph of function $\rho$, 
	\begin{equation}\label{e102}
		\Gamma=\{X\rho(X)\,:\,X\in\Omega\}.
	\end{equation}
Let $\mathcal{P}(\mu,\nu)$ be the set of measures on $\Omega\times\Omega^*$ with $\mu, \nu$ as their marginals. Let $\gamma\in\mathcal{P}(\mu,\nu)$. 
Denote by $C_+(\Omega)$ the set of positive continuous functions on $\Omega$. 
Define a functional
  \begin{equation}\label{e103}
   I(u,v)=\int_{\Omega\times\Omega^*} F(X,Y,u,v)d\gamma,
  \end{equation}
for $(u,v)\in C_+(\Omega)\times C_+(\Omega^*)$, where
  \begin{equation}\label{e104}
   F(X,Y,u,v)=f(X)\log u+g(Y)\left(\log v+\log (1-\frac{\langle X,Y\rangle}{v^{-1}+\sqrt{|Y|^2+v^{-2}}} )\right)
  \end{equation}
and $\langle\ , \rangle$ is the inner product in $\mathbb{R}^{n+1}$.
The main result is the following:

\begin{theorem}\label{t101}
Suppose that $f, g$ are two bounded positive functions on $\Omega, \Omega^*$, respectively, such that \eqref{e101} is satisfied. 
Suppose that $\Omega^*$ is contained in the cone $\mathcal{C}_V=\{tX\,:\,t>0,X\in V\}$ for a domain $V\subset\mathbb{S}^n$ and
	\begin{equation}\label{e105}
		\overline{\Omega}\cap\overline{V}=\emptyset,
	\end{equation}
where $\overline{\Omega}$ and $\overline{V}$ denote the closures of $\Omega$ and $V$, respectively.
Then there is a dual maximizing pair $(\rho,\eta)\in\mathcal{K}$, which satisfies
	\begin{equation*}
		I(\rho,\eta)=\sup_{(u,v)\in\mathcal{K}}I(u,v),
	\end{equation*} 
where $I(u,v)$ is given in \eqref{e103}--\eqref{e104}, and the constraint set $\mathcal{K}$ is given by
	\[\mathcal{K}=\left\{(u,v)\in C_+(\Omega)\times C_+(\Omega^*)\,:\,\phi(X,Y,u,v)\leq0\right\},\]
with the constraint function
	\begin{equation}\label{e106}
		\phi(X,Y,u,v)=\log u+\log v+\log\left(1-\frac{\langle X,Y\rangle}{v^{-1}+\sqrt{|Y|^2+v^{-2}}}\right).
	\end{equation}
Moreover, $\rho$ is a solution of the reflector problem with given densities $(\Omega, f)$ and $(\Omega^*, g)$.
\end{theorem}

In Theorem \ref{t101}, the functions $\rho, \eta$ are also called potential functions, and a solution of the reflector problem needs to be understood as a weak solution. The notion of weak solutions was introduced in \cite{KW,KO}, see \S 2.2 below. 
It follows from Remark \ref{r201} that for each choice of the parameter $c_0>0$, there is a weak solution $\rho$ satisfying $\inf_\Omega\rho\geq c_0$.

Moreover, we show that the function $\rho$ solves a Monge-Amp\`ere type equation. Assume that $\Omega^*$ is given implicitly by
	\begin{equation}\label{e107}
		\Omega^*=\{Z\in\mathbb{R}^{n+1}\,:\,\psi(Z)=0\}.
	\end{equation}
Suppose that $\Omega$ is a subset of upper unit sphere $\mathbb{S}^n_+=\mathbb{S}^n\cap\{x_{n+1}>0\}$. Let $X=(x,x_{n+1})$ be a parameterization of $\Omega$, where $x_{n+1}=\sqrt{1-|x|^2}=:\omega(x)$, and $x=(x_1,\cdots,x_n)$. For simplification, we define some auxiliary functions
	\begin{eqnarray}
		a \!\!&=&\!\! |D\rho|^2-(\rho+D\rho\cdot x)^2, \label{e108}\\
		b \!\!&=&\!\! |D\rho|^2+\rho^2-(D\rho\cdot x)^2, \label{e109}\\
		t \!\!&=&\!\! \frac{\rho x_{n+1}-y_{n+1}}{\rho x_{n+1}},\quad \beta=\frac{t}{(Y-X\rho)\cdot\nabla\psi}, \label{e110}
	\end{eqnarray}
and denote the matrix
	\begin{equation}\label{e111}
		\mathcal{N}=\{\mathcal{N}_{ij}\},\quad \mathcal{N}_{ij}=\delta_{ij}+\frac{x_ix_j}{1-|x|^2}.
	\end{equation}
By computing in this local orthonormal frame, we obtain our equation as follows	
\begin{theorem}\label{t102}
The function $\rho$ is a solution of
	\begin{equation}\label{e112}
		\left|\det\,\left[D^2\rho-\frac{2}{\rho}D\rho\otimes D\rho-\frac{a(1-t)}{2t\rho}\mathcal{N}\right]\right|=\left|\frac{a^{n+1}}{t^nb\beta}\right|\frac{f}{2^n\rho^{2n+1}\omega^2g|\nabla\psi|}.
	\end{equation}
\end{theorem}

The equation \eqref{e112} was previously obtained by Karakhanyan and Wang studying the near field reflector problem \cite{KW}.
One of the main differences in our derivation of \eqref{e112} is that instead of applying the reflection law as in \cite{KW}, we have differentiated the constraint function \eqref{e106} directly in general cases, see \eqref{e304} below. 
We remark that our method is more general and can be applied to the study of other reflector and refractor problems, \cite{GH,Liu}. 

This paper is arranged as follows. In Section 2, we first introduce a class of nonlinear optimization with potential functions, and then prove Theorem \ref{t101}. In Section 3, we derive the equation for potentials arising in general nonlinear optimization problems, and apply this formula to prove Theorem \ref{t102}. In Remark \ref{r301}, we point out that the far field reflector is a limit case of the near field one and related to a linear optimization problem.

\section{Formulation to optimization}

\subsection{Nonlinear optimization}
In general, we consider a functional 
	\begin{equation}\label{e201}
		I(u,v)=\int_{U\times V}F(x,y,u,v)d\gamma,
	\end{equation}
for $(u,v)\in C(U)\times C(V)$, where $U,V$ are two compact domains in $\mathbb{R}^n$ or a manifold $\mathcal{M}^n$, $F$ is a function on $U\times V\times\mathbb{R}^2$, and the measure $d\gamma$ has marginals $dx, dy$, which are the volume elements of $U, V$, respectively. 

We want to maximize the functional $I$ among all pairs $(u,v)$ in a constraint set 
	\begin{equation}\label{e202}
		\mathcal{K}=\left\{(u,v)\in C(U)\times C(V) \,:\, \phi(x,y,u,v)\leq0\mbox{ in }U\times V\right\},
	\end{equation}
where $\phi$ is the constraint function defined on $U\times V\times\mathbb{R}^2$.

Note that when 
	\begin{equation}\label{e203}
		F(x,y,u,v)=\frac{1}{|V|}f(x)u(x)+\frac{1}{|U|}g(y)v(y),\quad d\gamma=dxdy,
	\end{equation}
for some $f>0,\in L^1(U)$, $g>0,\in L^1(V)$ satisfying $\int_Uf=\int_Vg$, and
	\begin{equation}\label{e204}
		\phi(x,y,u,v)=u(x)+v(y)-c(x,y),
	\end{equation}
we have a linear optimization problem related to an optimal transportation with the cost function $c$, and mass densities $f, g$ supported on $U, V$, respectively. See \cite{Bre,GM,U,V}.

\begin{definition}
 A pair $(u,v)\in\mathcal{K}$ is called \emph{dual pair} with respect to $\phi$, if
	\begin{eqnarray}\label{e205}
		u(x) \!\!&=&\!\! \sup\{t\,:\,\phi(x,y,t,v(y))\leq0,\quad\forall y\in V\}, \\
		v(y) \!\!&=&\!\! \sup\{s\,:\,\phi(x,y,u(x),s)\leq0,\quad\forall x\in U\}. \nonumber
	\end{eqnarray}
If furthermore $I(u,v)=\sup_\mathcal{K}I$, $(u,v)$ is called \emph{dual maximizing pair} of $I$. In such a case, $u,v$ are also called \emph{potential functions} in the nonlinear optimization \eqref{e201}--\eqref{e202}. 
\end{definition}

Write $F=F(x,y,t,s)$ and $\phi=\phi(x,y,t,s)$, where $x,y,t,s$ are independent variables. 
Use the subscripts to denote the partial derivatives, i.e. $F_t=\partial F/\partial t$, $\phi_s=\partial\phi/\partial s$, etc. 
We always assume that $F$ is $C^1$ smooth in $t,s$ and integrable in $x,y$; $\phi$ is $C^2$ smooth in all variables.  
Moreover, we assume the following conditions on $F$ and $\phi$:
\begin{itemize}
	\item[(i)] $F(x,y,t,s)$ is monotone increasing in $t,s$, namely 
		\begin{equation}\label{e206}
			F_t\geq0,\quad F_s\geq0,\qquad \forall (x,y,t,s)\in U\times V\times\mathbb{R}\times\mathbb{R}.
		\end{equation} 
	\item[(ii)] $\phi(x,y,t,s)$ is strictly increasing in $t,s$, namely for a constant $\delta_0>0$
		\begin{equation}\label{e207}
			\phi_t\geq\delta_0,\quad \phi_s\geq\delta_0,\qquad \forall (x,y,t,s)\in U\times V\times\mathbb{R}\times\mathbb{R}.
		\end{equation} 
	\item[(iii)] for any pair $(u,v)\in\mathcal{K}$, the balance condition holds: 
		\begin{equation}\label{e208}
			\int_{U\times V}\left\{-F_t(x,y,u(x),v(y))+F_s\frac{\phi_t}{\phi_s}(x,y,u(x),v(y))\right\}d\gamma=0.
		\end{equation}
\end{itemize}

\begin{lemma}\label{l201}
Under the above assumptions and \eqref{e206}--\eqref{e208}, $I(u,v)$ in \eqref{e201} has a dual maximizing pair $(\bar u,\bar v)\in\mathcal{K}$, where $\mathcal{K}$ is the constraint set given in \eqref{e202}.
\end{lemma}

\begin{proof}
The proof is inspired by \cite{Bre,GM}. 
Given any pair $(u,v)\in\mathcal{K}$, we claim that $I(u,v)$ does not decrease if $v$ is replaced by
	\begin{equation}\label{e209}
		v^*(y) = \sup\{s\,:\,\phi(x,y,u(x),s)\leq0,\ \ \forall x\in U\}.
	\end{equation}

In fact, by the continuity of $\phi$ and $u$, for each $y\in V$ there is some $x\in\overline{U}$ such that 
	\[\phi(x,y,u(x),v^*(y))=0 \geq \phi(x,y,u(x),v(y)),\]
since $(u,v)\in\mathcal{K}$. By \eqref{e207}, $v^*\geq v$. Furthermore, $\phi(x,y,u(x),v^*(y))\leq0$ for all $(x,y)\in U\times V$, so $(u,v^*)\in\mathcal{K}$.	

Since $v^*\geq v$, by \eqref{e206} we have
	\[I(u,v^*)\geq I(u,v).\]
Similarly, if we define	
	\begin{equation}\label{e210}
	    u^{*}(x) = \sup\{t\,:\,\phi(x,y,t,v^*(y))\leq0,\ \ \forall y\in V\},
	\end{equation}
then $(u^*,v^{*})\in\mathcal{K}$ and
	\[I(u^*,v^{*})\geq I(u,v^*)\geq I(u,v).\]
Thus we do not decrease $I(u,v)$ by replacing $(u,v)$ by $(u^*,v^{*})$. The claim is proved. 

Define $\mathcal{K}_{C_0}=\mathcal{K}\cap\{u\geq C_0\}$, where $C_0$ is a constant. The constant $C_0$ may be chosen negative and sufficiently small in the following context. 
We show that $u^*$ and $v^*$ are uniformly bounded if $(u,v)\in\mathcal{K}_{C_0}$. 
Since $v^*\geq v, u\geq C_0$, by \eqref{e207} we have for each $y\in V$, $s:=v^*(y)$,
	\[\phi(x,y,C_0,s) \leq \phi(x,y,u(x),s) \leq0,\quad\mbox{ for all }x\in U.\]
Then by \eqref{e207} again, there exists a constant $C_1$ such that $s\leq C_1$. This implies that
	\begin{equation}\label{e211}
		v\leq v^*\leq C_1,
	\end{equation}
we may choose $C_1$ such that $\sup_V v^*=C_1$. By a similar argument, there is another constant $\tilde C_0$ depending on $\phi$ and $C_1$ such that $\inf_U u^*=\tilde C_0$. The constant $\tilde C_0\geq C_0$, since $u^*\geq u$ in $U$, and so $(u^*,v^*)\in\mathcal{K}_{C_0}$.

We next deduce the lower bound of $v^*$ and the upper bound of $u^*$ by showing that $u^*$ and $v^*$ are locally Lipschitz functions. 
Consider two points in $U$, $x_1\neq x_2$ and $|x_1-x_2|<\varepsilon$ is sufficiently small. There are two points $y_1,y_2\in\overline{V}$ such that
	\begin{eqnarray*}
		\phi(x_1,y_1,u^*(x_1),v^*(y_1))\!\!&=&\!\!0, \\
		\phi(x_2,y_2,u^*(x_2),v^*(y_2))\!\!&=&\!\!0.
	\end{eqnarray*}
Then we have
	\[\begin{split}
		0=&\phi(x_2,y_2,u^*(x_2),v^*(y_2))-\phi(x_1,y_2,u^*(x_1),v^*(y_2)) \\
		  &+\phi(x_1,y_2,u^*(x_1),v^*(y_2))-\phi(x_1,y_1,u^*(x_1),v^*(y_1)) \\
		 =&\phi_t(\hat x,y_2,\hat u^*,v^*)(u^*(x_2)-u^*(x_1))-\phi_x(\hat x,y_2,\hat u^*,v^*)\cdot(x_2-x_1) \\
		  &+\phi(x_1,y_2,u^*(x_1),v^*(y_2)),
	\end{split}\]
where $\hat u^*=\theta u^*(x_1)+(1-\theta)u^*(x_2), \hat x=\bar\theta x_1+(1-\bar\theta)x_2$, for some $\theta,\bar\theta\in(0,1)$.
Noting that $\phi(x_1,y_2,u^*(x_1),v^*(y_2))\leq0$, we have 
	\[u^*(x_2)-u^*(x_1)\geq -\frac{C_2}{C_3}|x_2-x_1|,\]
where the constants $C_2=\sup(|\partial_x\phi|+|\partial_y\phi|)$, and $C_3=\min\left\{\inf\partial_t\phi, \inf\partial_s\phi\right\}$.
Due to \eqref{e207}, the constant $C_3\geq\delta_0$ is positive.
On the other hand, replacing $\phi(x_1,y_2,u^*(x_1),v^*(y_2))$ by $\phi(x_2,y_1,u^*(x_2),v^*(y_1))$ in the above calculation, we have
	\[u^*(x_2)-u^*(x_1)\leq \frac{C_2}{C_3}|x_2-x_1|.\]
Therefore, the Lipschitz constant of $u^*$ on $U$ is controlled by
	\begin{equation}\label{e212}
		\|u^*\|_{Lip(U)}\leq C_4,
	\end{equation}
where the constant $C_4=C_2/C_3$.
A similar argument holds for $v^{*}$ as well, which implies that $\|v^*\|_{Lip(V)}\leq C_4$.
Hence, we have $u^*\leq\tilde C_0+C_4\diam(U)$ and $v^*\geq C_1-C_4\diam(V)$ because of \eqref{e211}.

We conclude, therefore, that any pair $(u,v)\in\mathcal{K}_{C_0}$ may be replaced by a bounded, Lipschitz pair $(u^*,v^*)\in\mathcal{K}_{C_0}$ without decreasing $I$. We now choose a sequence $\{(u_k,v_k)\}\subset\mathcal{K}_{C_0}$ such that
	\[I(u_k,v_k)\rightarrow \sup_{(u,v)\in\mathcal{K}_{C_0}}I(u,v).\]
By the above considerations we may assume that each $(u_k,v_k)$ is a bounded, uniformly Lipschitz pair, uniformly with respect to $k$, so there is a subsequence converging uniformly to a bounded, Lipschitz, maximizing pair $(\bar u,\bar v)\in\mathcal{K}_{C_0}$. 

Last, we show that when $C_0<0$ is sufficiently small, 
	\[\sup_{(u,v)\in\mathcal{K}_{C_0}}I(u,v)=\sup_{(u,v)\in\mathcal{K}}I(u,v),\]
or equivalently, $\sup_{\mathcal{K}_{C_0}}I$ is independent of $C_0$. 
By definition, one has $\sup_{\mathcal{K}_{C_0-1}}I\geq\sup_{\mathcal{K}_{C_0}}I$. So, it suffices to show the reverse inequality.
Let $(u,v)\in\mathcal{K}_{C_0-1}$ be a maximizer such that $I(u,v)=\sup_{\mathcal{K}_{C_0-1}}I$, and $\{x_k\}_{k=1,\cdots,N}$ be a set of points in $U$. 
For a small constant $\varepsilon>0$, define
	\[\tilde u=\left\{\begin{array}{ll}
	 u & \mbox{ in } U-\cup_NB_\varepsilon(x_k), \\
	 u+2 & \mbox{ in } \cup_NB_\varepsilon(x_k).
	\end{array}\right.\]
Note that we may replace $\tilde u$ by its mollification $\tilde u_h=\rho_h*\tilde u$, where $\rho_h$ is the standard mollifier function \cite{GT}. For simplicity, we assume $\tilde u$ continuous in the sense that for $h>0$ sufficiently small, 
	\[I(\tilde u_h, v)=I(u,v)+O(N\varepsilon^n).\]
Define
	\begin{eqnarray*}
		\tilde v^*(y) \!\!&=&\!\! \sup\{s\,:\,\phi(x,y,\tilde u(x),s)\leq0,\ \ \forall x\in U\}, \\
		\tilde u^*(x) \!\!&=&\!\! \sup\{t\,:\,\phi(x,y,t,\tilde v^*(y))\leq0,\ \ \forall y\in V\}.
	\end{eqnarray*}
Since the constraint function $\phi$ is smooth and by \eqref{e207}, except a set $E\subset U$ and a set $E'\subset V$ of measure $|E|=|E'|=O(N\varepsilon^n)$,
	\begin{eqnarray*}
		\tilde v^*\!\!&=&\!\! v-2\frac{\phi_t}{\phi_s}+O(\delta)\quad\mbox{ in }V\setminus E', \\
	    \tilde u^*\!\!&=&\!\! u+2+O(\delta)\quad\mbox{ in }U\setminus E,
	\end{eqnarray*}
where $\delta:=\min_{i\neq j}\{\dist(x_i,x_j)\}$. 
Therefore, by \eqref{e208} and the mean value theorem we have
	\[\begin{split}
		I(\tilde u^*,\tilde v^*) &= I(u,v)+2\int_{(U\setminus E)\times(V\setminus E')}\left\{F_t-F_s\frac{\phi_t}{\phi_s}\right\}d\gamma+O(\delta)+O(N\varepsilon^n) \\
			&\geq I(u,v)-C\delta-CN\varepsilon^n.
	\end{split}\]
As $(u,v)\in\mathcal{K}_{C_0-1}$, we may assume that $\inf_U u=C_0-1$. Otherwise, one has $\inf_U u=C_0-\tau_0$ for some constant $\tau_0<1$. This implies that  
$\sup_{\mathcal{K}_{C_0-1}}I=\sup_{\mathcal{K}_{C_0-\tau_0}}I$, namely $\sup_{K_{C_0}}I$ is independent of $C_0$, and the proof is finished.
By the definition, $\delta$ will become small if the number of points $N$ is sufficiently large so that we have $(\tilde u^*,\tilde v^*)\in\mathcal{K}_{C_0}$ and
	\[\sup_{\mathcal{K}_{C_0}}I\geq I(\tilde u^*,\tilde v^*)\geq \sup_{\mathcal{K}_{C_0-1}}I-C\delta-CN\varepsilon^n.\]
Then, choosing $\varepsilon>0$ sufficiently small we have
	\[\sup_{\mathcal{K}_{C_0-1}}I\leq\sup_{\mathcal{K}_{C_0}}I,\]
by letting $\delta\to0, \varepsilon\to0$, which implies that $\sup_{\mathcal{K}_{C_0}}I$ is independent of $C_0$, and the proof is finished.
\end{proof}

\begin{remark}\label{r201}
From the proof of Lemma \ref{l201}, we conclude that there exist infinitely many maximizing pairs. In fact, if $(u,v)$ is a maximizer and $C_0=\inf_U u$, then there is another maximizer in $\mathcal{K}_{C_0+1}$, which is different from $(u,v)$.
\end{remark}

\begin{lemma}\label{l202}
 Let $(u,v)\in\mathcal{K}$ be a dual maximizing pair in Lemma \ref{l201}. The equation
	\begin{equation}\label{e213} 
		\phi(x,T(x),u(x),v(T(x)))=0
	\end{equation}
 can be solved by a mapping $T : U\to V$ implicitly determined by the formula
	\begin{equation}\label{e214}
		\phi_x(x,T(x),u,v)+\phi_t(x,T(x),u,v)Du(x)=0,
	\end{equation}
 at any differentiable point of $u$.
 Furthermore, if for any $(x,y,t,s)\in U\times V\times\mathbb{R}^2$,
	\begin{equation}\label{e215}
	 \det\,\left[\phi_{xy}+\phi_{yt}\otimes Du+\phi_{xs}\otimes Dv+\phi_{ts}Du\otimes Dv\right]\neq0,
	\end{equation}
 the mapping $T$ is uniquely determined by \eqref{e214}.  
\end{lemma}
The mapping $T$ in Lemma \ref{l202} is called the \emph{optimal mapping} associated to the dual maximizing pair $(u,v)$.
The inequality \eqref{e215} is a generalization of (A2) condition in optimal transportation \cite{MTW}.

\begin{proof}
Since $u$ satisfies \eqref{e205} and $v, \phi$ are continuous, for each $x\in U$, there exists some $y=:T(x)\in\overline{V}$ such that
	\begin{eqnarray}\label{e216}
		\phi(x,y,u(x),v(y)) \!\!&=&\!\! 0, \\
		\phi(x',y,u(x'),v(y)) \!\!&\leq&\!\! 0, \nonumber
	\end{eqnarray}
for any other $x'\in U$. Let $x\in U$ be a differentiable point of $u$, by differentiation we have
	\begin{equation*}
		\phi_x(x,y,u,v)+\phi_t(x,y,u,v)Du(x)=0.
	\end{equation*}
If there exists another $\tilde y\neq y$ in $\overline{V}$ such that
  \[\phi_x(x,\tilde y,u,v)+\phi_t(x,\tilde y,u,v)Du(x)=0.\]
By the mean value theorem,
  \[(\phi_{xy}+\phi_{yt}\otimes Du+\phi_{xs}\otimes Dv+\phi_{ts}Du\otimes Dv)\cdot(\tilde y-y)=0,\]
where the matrix is valued at $(x,\hat y,u(x),v(\hat y))$ with $\hat y=\alpha\tilde y+(1-\alpha)y$ for some $\alpha\in(0,1)$. This is a contradiction with \eqref{e215}, since $\tilde y\neq y$. 
\end{proof}

Moreover, such an obtained optimal mapping $T$ satisfies the following property, which is a kind of conservation of energy.

\begin{lemma}\label{l203}
Let $T$ be the optimal mapping associated to a dual maximizing pair $(u,v)$.
Assume the constraint function $\phi=\phi(x,y,t,s)$ is smooth and satisfies \eqref{e207}.
Then for any $h\in C(V)$, there holds
	\begin{equation}\label{e217}
		0=\int_{U\times V}\left\{-F_t\frac{\phi_s}{\phi_t}h(T(x))+F_s h(y)\right\}d\gamma.
	\end{equation}
	
\end{lemma}

\begin{proof}
Let $h\in C(V)$ and $|\epsilon|<1$ sufficiently small. Define
	\begin{equation}\label{e218}
		v_\epsilon(y)=v(y)+\epsilon h(y)
	\end{equation}
and
	\begin{equation}\label{e219}
		u_\epsilon(x)=\sup\{t\,:\,\phi(x,y,t,v_\epsilon(y))\leq0,\ \ \forall y\in V\}.
	\end{equation}
Then $(u_\epsilon,v_\epsilon)\in\mathcal{K}$ and $(u_0,v_0)=(u,v)$. 

Since $(u,v)$ satisfies \eqref{e205}, by Lemma \ref{l202}, for every $x\in U$ the supremum \eqref{e205} is attained at point $y_0=T(x)$. We claim that at these points we have
	\begin{equation}\label{e220}
		u_\epsilon(x)-u(x)=-\epsilon\frac{\phi_s}{\phi_t}h(T(x))+o(\epsilon).
	\end{equation}
	
To prove \eqref{e220}, first we show that $LHS\leq RHS$.
	\[\begin{split}
		0&=\phi(x,y_0,u(x),v(y_0)) \\
		 &=\phi(x,y_0,u(x),v_\epsilon(y_0)-\epsilon h(y_0)) \\
		 &=\phi(x,y_0,u(x),v_\epsilon(y_0))-\epsilon\phi_sh(y_0)+o(\epsilon) \\
		 &=\phi(x,y_0,u_\epsilon(x)+u(x)-u_\epsilon(x),v_\epsilon(y_0))-\epsilon\phi_sh(y_0)+o(\epsilon) \\
		 &=\phi(x,y_0,u_\epsilon(x),v_\epsilon(y_0))+\phi_t(u(x)-u_\epsilon(x))-\epsilon\phi_sh(y_0)+o(\epsilon) \\
	  &\leq\phi_t(u(x)-u_\epsilon(x))-\epsilon\phi_sh(y_0)+o(\epsilon).
	\end{split}\]
By \eqref{e207} we have
	\[u_\epsilon(x)-u(x) \leq -\epsilon\frac{\phi_s}{\phi_t}h(y_0)+o(\epsilon).\]
	
To show $LHS\geq RHS$ we use the fact that for any such $x\in U$ there are points $y_\epsilon\in\overline V$ such that the supremum in \eqref{e219} is attained. Thus
	\[\begin{split}
		0&\geq \phi(x,y_\epsilon,u(x),v(y_\epsilon)) \\
		 &=\phi(x,y_\epsilon,u(x),v_\epsilon(y_\epsilon)-\epsilon h(y_\epsilon)) \\
		 &=\phi(x,y_\epsilon,u(x),v_\epsilon(y_\epsilon))-\epsilon\phi_sh(y_\epsilon)+o(\epsilon) \\
		 &=\phi(x,y_\epsilon,u_\epsilon(x)+u(x)-u_\epsilon(x),v_\epsilon(y_\epsilon))-\epsilon\phi_sh(y_\epsilon)+o(\epsilon) \\
		 &=\phi(x,y_\epsilon,u_\epsilon(x),v_\epsilon(y_\epsilon))+\phi_t(u(x)-u_\epsilon(x))-\epsilon\phi_sh(y_\epsilon)+o(\epsilon) \\
	     &=\phi_t(u(x)-u_\epsilon(x))-\epsilon\phi_sh(y_\epsilon)+o(\epsilon).
	\end{split}\]
Then by \eqref{e207} we have
	\[\begin{split}
		u_\epsilon(x)-u(x) &\geq -\epsilon\frac{\phi_s}{\phi_t}h(y_\epsilon)+o(\epsilon) \\
			&=  -\epsilon\frac{\phi_s}{\phi_t}h(y_0)+\epsilon\frac{\phi_s}{\phi_t}\left(h(y_0)-h(y_\epsilon)\right)+o(\epsilon).
	\end{split}\]
Since the supremum in \eqref{e205} is attained at $y_0$, we have $y_\epsilon\to y_0$ as $\epsilon\to0$, and therefore, since $h\in C(V)$,
	\[\epsilon\frac{\phi_s}{\phi_t}\left(h(y_0)-h(y_\epsilon)\right)=o(\epsilon).\]
This implies that $LHS\geq RHS$, and \eqref{e220} follows.
	
Next, since $(u,v)=(u_0,v_0)$ maximizes $I$ among all pairs in $\mathcal{K}$, we obtain
	\[\begin{split}
		0&=\lim_{\epsilon\to0}\frac{I(u_\epsilon,v_\epsilon)-I(u,v)}{\epsilon} \\
		 &=\lim_{\epsilon\to0}\int_{U\times V}\frac{F(x,y,u_\epsilon,v_\epsilon)-F(x,y,u,v)}{\epsilon}d\gamma \\
		 &=\int_{U\times V}\left\{-F_t\frac{\phi_s}{\phi_t}h(T(x))+F_s h(y)\right\}d\gamma.
	\end{split}\]
\end{proof}

\subsection{Formulation of reflector problem}
In order to formulate the near field reflector problem to an optimization problem, we need the notion of \emph{ellipsoid of revolution}, which has a special reflection property: the light rays from one focus are always reflected to the other focus.

In the polar coordinate system, an {ellipsoid of revolution} $E=E(Y,p)$ with one focus at the origin, the other focus at $Y$, and \emph{focal parameter} $p\in(0,\infty)$ can be represented as $E=\{X\rho_e(X)\,:\,X\in\mathbb{S}^n\}$ by a radial function
	\begin{equation}\label{e221}
		\rho_e(X)=\frac{p}{1-\epsilon(p)\langle X,\frac{Y}{|Y|}\rangle}
	\end{equation}
and
	\begin{equation}\label{e222}
		\epsilon(p)=\sqrt{1+\frac{p^2}{|Y|^2}}-\frac{p}{|Y|}
	\end{equation}
is the \emph{eccentricity}, see \cite{KO}. Note that any such ellipsoid is uniquely determined by $Y$ and $p$. If we regard $p=p(Y)$ as a \emph{focal function} on $\Omega^*$, we then have a family of ellipsoids.

We recall that \cite{KW}, for an \emph{admissible} reflector $\Gamma_\rho$, at each point $X\rho(X)\in\Gamma_\rho$ there exists a \emph{supporting ellipsoid}, namely, for some $Y\in\Omega^*$
	\begin{equation}\label{e223}
	\left\{\begin{array}{rl}
	  \rho(X) \!\!\!&= \frac{p(Y)}{1-\epsilon(p(Y))\langle X,\frac{Y}{|Y|}\rangle},\\
	  \rho(X')\!\!\!&\leq \frac{p(Y)}{1-\epsilon(p(Y))\langle X',\frac{Y}{|Y|}\rangle},\ \ \forall X'\in\Omega.
	\end{array}\right.
	\end{equation}
In the following context, we also say $\rho$ is admissible if $\Gamma_\rho$ is admissible. 

Next we define a set-valued mapping $T_\rho:\Omega\to\Omega^*$. For any $X\in\Omega$,
  \begin{equation}\label{e224}
    \begin{array}{rl}
     T_\rho(X)=\{Y\in\Omega^*\,:\!&\!\!Y \mbox{ is the focus of }\\ &\mbox{a supporting ellipsoid of $\Gamma_\rho$ at $X\rho(X)$}\}.
    \end{array}
  \end{equation}
Note that at any differentiable point $X$ of $\rho$, $T_\rho(X)$ is single valued and is exactly the reflection mapping.
For any subset $G\subset\Omega$, we denote $T_\rho(G)=\bigcup_{X\in G}T_\rho(X)$. Therefore, we can define a measure $\mu_\#=\mu_{\rho,g}$ in $\Omega$ such that for any Borel set $G\subset\Omega$,
  \begin{equation}\label{e225}
   \mu_\#(G)=\int_{T_\rho(G)}gd\nu.
  \end{equation}

\begin{definition}
  An admissible function $\rho$ is called a weak solution of the reflector problem if $\mu_{\rho,g}=fd\mu$ as measures, namely for any Borel set $G\subset\Omega$,
  \begin{equation}\label{e226}
   \int_Gfd\mu=\int_{T_\rho(G)}gd\nu.
  \end{equation}
\end{definition}
The above definition was introduced in \cite{KW}. Obviously an admissible smooth solution is a weak solution, in that case, the reflector $\Gamma_\rho$ is naturally an envelope of a family of confocal ellipsoids of revolution. Therefore, the radial function $\rho$ satisfies
	\begin{equation}\label{e227}
		\rho(X)=\inf_{Y\in\Omega^*}\frac{p(Y)}{1-\epsilon(p(Y))\langle X,\frac{Y}{|Y|}\rangle},\quad X\in\Omega,
	\end{equation}
and for each $Y\in\Omega^*$ the ellipsoid $E_{Y,p(Y)}$ is supporting to $\Gamma_\rho$, we also have the focal function $p$ satisfies
	\begin{equation}\label{e228}
		p(Y)=\sup_{X\in\Omega}\rho(X)\left[1-\epsilon(p(Y))\langle X,\frac{Y}{|Y|}\rangle\right],\quad Y\in\Omega^*.
	\end{equation}
Note that in \eqref{e227} for each $X\in\Omega$ the infimum is achieved at some $Y\in\Omega^*$ and in \eqref{e228} for each $Y\in\Omega^*$ the supremum is achieved at some $X\in\Omega$.

The relations \eqref{e227}--\eqref{e228} between the radial and focal functions of a reflector $\Gamma_\rho$ are analogous to the classical relations between the radial and support functions for convex bodies, for example, see \cite{Sch}. 
Inspired by that and \cite{W04}, we set $\eta=1/p$. Then the pair $(\rho,\eta)$ satisfies the following dual relation
	\begin{eqnarray}\label{e229}
		\rho(X)\!\!&=&\!\!\inf_{Y\in\Omega^*}\frac{1}{\eta(Y)\left(1-\epsilon(\eta(Y))\langle X,\frac{Y}{|Y|}\rangle\right)}, \\
		\eta(Y)\!\!&=&\!\!\inf_{X\in\Omega}\frac{1}{\rho(X)\left(1-\epsilon(\eta(Y))\langle X,\frac{Y}{|Y|}\rangle\right)}, \nonumber
	\end{eqnarray}
where $\eta$ is a Legendre type transform of $\rho$, \cite{GW}.

Similarly to \cite{W04}, we can now formulate the reflector problem to a nonlinear optimization \eqref{e201}--\eqref{e202} as follows. Set the functional
	\begin{equation}\label{e230}
	\begin{split}
		I(\rho,\eta)&=\int_{\Omega\times\Omega^*} F(X,Y,\rho,\eta) \\
			&=\int_{\Omega\times\Omega^*}f(X)\log\rho+g(Y)\left(\log\eta+\log (1-\frac{\langle X,Y\rangle}{\eta^{-1}+\sqrt{|Y|^2+\eta^{-2}}} )\right),
	\end{split}
	\end{equation}
and the constraint set
	\[\mathcal{K}=\left\{(\rho,\eta)\in C_+(\Omega)\times C_+(\Omega^*)\,:\,\phi(X,Y,\rho,\eta)\leq0\right\},\]
with the constraint function
	\begin{equation}\label{e231}
		\phi(X,Y,\rho,\eta)=\log\rho+\log\eta+\log\left(1-\epsilon(\eta(Y))\langle X,\frac{Y}{|Y|}\rangle\right).
	\end{equation}
In fact, by \eqref{e222} and $\eta=1/p$ it is easy to see that
  \begin{equation}\label{e232}
   \epsilon(\eta(Y))\langle X,\frac{Y}{|Y|}\rangle=\frac{\langle X,Y\rangle}{\eta^{-1}+\sqrt{|Y|^2+\eta^{-2}}}.
  \end{equation}

\begin{lemma}\label{l204}
 Let $(\rho,\eta)$ be a dual maximizing pair of \eqref{e230}--\eqref{e231}, and $T$ be the associated optimal mapping. Then $T=T_\rho$ at any differentiable point of $\rho$, where $T_\rho$ is the reflection mapping in \eqref{e224}.
\end{lemma}

\begin{proof}
We first introduce some geometric notation. 
By restricting to a subset we may assume that $\Omega$ is in the north hemisphere. Let $X=(x,x_{n+1})$ be a smooth parameterization of $\Omega\subset\mathbb{S}^n$, where $x_{n+1}=\sqrt{1-|x|^2}$ and $x=(x_1,\cdots,x_n)$.

Denote $\partial_i=\partial/\partial x_i$, $e_i=\partial_iX$, and the metric $g_{ij}=\langle e_i,e_j\rangle$, where $\langle\, ,\rangle$ is the inner product of $\mathbb{R}^{n+1}$. By direct computations, for $i,j,k,l=1,\cdots,n$,
	\begin{eqnarray}
		&&e_i=\left(0,\cdots,1,\cdots,0,\frac{-x_i}{\sqrt{1-|x|^2}}\right),\qquad \mbox{$1$ is in the $i$th coordinate}, \label{e233} \\
		&&g_{ij}=\delta_{ij}+\frac{x_ix_j}{1-|x|^2}, \qquad		g^{ij} = \delta_{ij}-x_ix_j,	\quad\mbox{where $(g^{ij})=(g_{ij})^{-1}$}, \label{e234}
	\end{eqnarray}
and the Christoffel symbols are
	\begin{equation}\label{e235}
		\Gamma_{ij}^k=\frac{1}{2}g^{kl}\left(\partial_ig_{jl}+\partial_jg_{il}-\partial_lg_{ij}\right)=x_k\left(\delta_{ij}+\frac{x_ix_j}{1-|x|^2}\right).
	\end{equation}
Denote $e_{n+1}=-X$, the unit inner normal of $\mathbb{S}^n$ at $X$. The Gauss formula is
	\[\partial_je_i=\Gamma_{ij}^ke_k+h_{ij}e_{n+1},\]
where the second fundamental form
	\begin{equation}\label{e236}
	\begin{split}
		h_{ij} &= \delta_{ij}+\frac{x_ix_j}{1-|x|^2}\\
		       &=g_{ij}.
	\end{split}
	\end{equation}
Namely, one has that
	\begin{equation}\label{e237}
		\partial_je_i = \Gamma_{ij}^ke_k-g_{ij}X.
	\end{equation}
The above equalities \eqref{e233}--\eqref{e237} can all be obtained by basic computations.

Let $\rho$ be a function defined on $\Omega$. The tangential gradient of $\rho$ is defined by
	\begin{equation}\label{e238}
		\nabla\rho=\sum_{i,j=1}^ng^{ij}e_i\partial_j\rho.
	\end{equation}
Note that $\nabla\rho(X)\in T_X\mathbb{S}^n$, the tangent space, so $\langle\nabla\rho,X\rangle=0$. By direct calculation
	\begin{equation}\label{e239}
	\begin{split}
		\langle\nabla\rho,e_i\rangle &= \langle g^{jk}e_j\partial_k\rho,e_i\rangle \\
				&= g^{jk}g_{ij}\partial_k\rho=\delta_{ik}\partial_k\rho=\partial_i\rho,
	\end{split}
	\end{equation}
for all $1\leq i\leq n$. Let $D\rho=(\partial_1\rho,\cdots,\partial_n\rho)$ be the standard gradient of $\rho$. From \eqref{e233}, \eqref{e234} and \eqref{e238}, we have
	\begin{eqnarray}
		\nabla\rho \!\!&=&\!\! (D\rho,0)-\left(D\rho\cdot x\right)X, \label{e240} \\
		|\nabla\rho|^2 \!\!&=&\!\! \langle\nabla\rho,\nabla\rho\rangle = |D\rho|^2-\left(D\rho\cdot x\right)^2. \label{e241}
	\end{eqnarray}
	
Let $\Gamma_\rho=\{X\rho(X)\,:\,X\in\Omega\}$ be the graph of $\rho$ over $\Omega$. We claim that the unit normal of $\Gamma_\rho$ at $X\rho(X)$ is
	\begin{equation}\label{e242}
		\gamma=\frac{\nabla\rho-\rho X}{\sqrt{\rho^2+|\nabla\rho|^2}}.
	\end{equation}
Indeed, for $i=1,\cdots,n$, the tangential of $\Gamma_\rho$ at $X\rho(X)$ is
	\[\tau_i=\partial_i(X\rho(X))=\rho e_i+(\partial_i\rho) X.\]
From \eqref{e239}, for any $i=1,\cdots,n$, the following holds:
	\[\begin{split}
		\langle\tau_i,\gamma\rangle &= \frac{1}{\sqrt{\rho^2+|\nabla\rho|^2}}\langle \rho e_i+(\partial_i\rho)X,\nabla\rho-\rho X\rangle \\
			&= \frac{1}{\sqrt{\rho^2+|\nabla\rho|^2}}(\rho\partial_i\rho-\rho\partial_i\rho)=0.
	\end{split}\]
It is obvious that $|\gamma|=1$, thus $\gamma$ is the unit normal. 	

At the differentiable point $X$ of $\rho$, by \eqref{e224}, $Y=T_\rho(X)$ is the focus of the supporting ellipsoid of $\Gamma_\rho$ at $X\rho(X)$. Denote the reflected direction by $Y_r=\frac{Y-X\rho}{|Y-X\rho|}$. By \eqref{e242} and the reflection law,
  \begin{equation}\label{e243}
  \begin{split}
   Y_r &= X-2\langle X,\gamma\rangle\gamma \\
       &= \frac{2\rho\nabla\rho+(|\nabla\rho|^2-\rho^2)X}{|\nabla\rho|^2+\rho^2}.
  \end{split}
  \end{equation}
Denote the length of reflected ray
	\begin{equation}\label{e244}
		d:=|Y-X\rho|.
	\end{equation}
Hence, we have
	\begin{equation}\label{e245}
	\begin{split}
	 Y &= T_\rho(X)=  X\rho+Y_rd \\
		  &= \frac{2\rho\nabla\rho+(|\nabla\rho|^2-\rho^2)X}{|\nabla\rho|^2+\rho^2}d+X\rho,
	\end{split}
	\end{equation}	  
and
	\begin{equation}\label{e246}
		\langle X,Y\rangle =d\frac{|\nabla\rho|^2-\rho^2}{|\nabla\rho|^2+\rho^2}+\rho.
	\end{equation}	

On the other hand, by differentiating the constraint function in \eqref{e231} and the formula \eqref{e214}, we obtain
  \begin{equation}\label{e247}
   \frac{\partial_i\rho}{\rho}=\frac{\epsilon\langle e_i,Y_e\rangle}{1-\epsilon\langle X,Y_e\rangle},
  \end{equation}
where $Y_e=T(X)/|T(X)|$, $T(X)=Y$ is the optimal mapping. By noting that $e_i\perp X$, $\nabla\rho\perp X$ and $\langle\nabla\rho,e_i\rangle=\partial_i\rho$, from \eqref{e247} we have the decomposition
  \begin{equation}\label{e248}
   Y_e=\frac{1-\epsilon\langle X,Y_e\rangle}{\epsilon\rho}\nabla\rho+\langle X,Y_e\rangle X.
  \end{equation}

From \eqref{e213}, \eqref{e229} and \eqref{e231}, observe that at differentiable point $X$ of $\rho$, there exists a unique supporting ellipsoid $E$ of $\Gamma_\rho$ at $X\rho(X)$, with foci $O, Y$ and eccentricity $\epsilon$. 
Note that the sum of length $\rho=|X\rho(X)-O|$ and length $d=|Y-X\rho(X)|$ equals to the diameter of $E$, i.e.
	\begin{equation}\label{e249}
		\rho+d=\diam(E).
	\end{equation}
By the definition of eccentricity $\epsilon$,
	\begin{equation}\label{e250}
		\epsilon=\frac{|Y|}{\diam(E)}=\frac{|Y|}{\rho+d}.
	\end{equation}
Combining \eqref{e248} and \eqref{e250}, one obtains the following equation for $Y=T(X)$,	 
  \begin{equation}\label{e251}
   Y=\frac{\rho+d-\langle X,Y\rangle}{\rho}\nabla\rho+\langle X,Y\rangle X.
  \end{equation}
It then suffices to show that $Y=T_\rho(X)$ in \eqref{e245} is a solution of \eqref{e251}. In fact, by \eqref{e246} we have
  \begin{equation}\label{e252}
  \begin{split}
   T(X) &= \left(d-d\frac{|\nabla\rho|^2-\rho^2}{|\nabla\rho|^2+\rho^2}\right)\nabla\rho/\rho+X\rho+Xd\frac{|\nabla\rho|^2-\rho^2}{|\nabla\rho|^2+\rho^2} \\
       &= X\rho+\frac{2\rho\nabla\rho+(|\nabla\rho|^2-\rho^2)X}{|\nabla\rho|^2+\rho^2}d \\
       &= X\rho+Y_rd = T_\rho(X).
  \end{split}
  \end{equation}
\end{proof}

\begin{proof}[Proof of Theorem \ref{t101}]
The proof essentially follows from \cite{W04}. 
Let $u=\log\rho$, $v=\log\eta$. In order to apply Lemma \ref{l201}, we need first to verify that $F$ and $\phi$ in \eqref{e230}--\eqref{e231} satisfy the conditions \eqref{e206}--\eqref{e208}.
For the constraint function $\phi$ in \eqref{e231}, it is easy to see that $\phi_t=1>0$. By \eqref{e222} and $\eta=1/p$,
  \[\begin{split}
     \phi_s &= \eta\left(\frac{1}{\eta}-\frac{\frac{\partial\epsilon}{\partial\eta}\langle X,Y_e\rangle}{1-\epsilon\langle X,Y_e\rangle}\right) \\
	    &= 1-\frac{1}{\sqrt{1+\eta^2|Y|^2}}\frac{\epsilon\theta}{1-\epsilon\theta},
    \end{split}\]
where $\theta=\langle X,Y_e\rangle\in[-1,1)$ due to \eqref{e105}. Since $\phi_s$ is decreasing in $\theta$, we have
  \[\phi_s>1-\frac{1}{\sqrt{1+\eta^2|Y|^2}}\frac{\epsilon\theta_0}{1-\epsilon\theta_0},\]
where the constant $\theta_0<1$ depends on domains $\Omega,\Omega^*$.
Set $\tau=1/(\eta|Y|)$, $\epsilon=\sqrt{1+\tau^2}-\tau$. One has the second term in the above inequality
  \[\frac{1}{\sqrt{1+\eta^2|Y|^2}}\frac{\epsilon\theta_0}{1-\epsilon\theta_0} < \frac{\tau(\sqrt{1+\tau^2}-\tau)}{\sqrt{1+\tau^2}(1-\sqrt{1+\tau^2}+\tau)}=:h(\tau),\]
where the function $h$ is decreasing in $\tau$. Thus, by the Taylor expansion of $\sqrt{1+\tau^2}$ near $\tau=0$,
  \[h(\tau)<\lim_{\tau\to0}h(\tau)=1,\quad\mbox{for }\tau>0.\]
Hence, we otain $\phi_s>\delta_0$ and \eqref{e207} holds, for a positive constant $\delta_0$. From \eqref{e230}--\eqref{e232}, one can see that $F_t=\phi_tf(X)$ and $F_s=\phi_sg(Y)$. Since $f$ and $g$ are both positive, we have the condition \eqref{e206} satisfied. 
The condition \eqref{e208} is an equivalent to the assumption \eqref{e101}.

Therefore, from Lemma \ref{l201}, we obtain a dual maximizing pair $(\rho,\eta)\in\mathcal{K}$ in Theorem \ref{t101}. 
Then by the dual relation \eqref{e205} and \eqref{e231}, one can see that $\rho$ is admissible \eqref{e223}, and $\eta$ is the Legendre type transform of $\rho$ as in \eqref{e229}. 
From Lemma \ref{l201}, one knows that $\rho$ is Lipschitz continuous. Actually, since an admissible function has supporting ellipsoid at any point of its graph, it is semi-convex and twice differentiable almost everywhere \cite{KW}. Hence, by Lemma \ref{l204} $T_\rho=T$ a.e., where $T_\rho$ is the mapping defined in \eqref{e224}.

Next, we show that $T$ satisfies the mearsure preserving condition \eqref{e226}.
Since $F_t=\phi_tf(X)$ and $F_s=\phi_sg(Y)$, by \eqref{e207} and applying Lemma \ref{l203} to $T$, we obtain that
  \[\int_\Omega f(X)h(T(X))d\mu=\int_{\Omega^*}g(Y)h(Y)d\nu,\]
for arbitrary test functions $h\in C(\Omega^*)$.
Therefore, since $T_\rho=T$ a.e., we see that $T_\rho$ satisfies \eqref{e226}, namely $\rho$ is a weak solution of the reflector problem.
\end{proof}

\section{Derivation of equation}

We first derive the partial differential equation for the nonlinear optimization problem \eqref{e201}--\eqref{e202} in general.
Let $(u,v)$ be a dual maximizing pair of $I$. Assume that all the functions are smoothly differentiable at this stage. By a second differentiation of \eqref{e214} we obtain 
	\begin{equation}\label{e301}
	\begin{split}
		0=& \phi_{xx}+\phi_{xy}DT+2\phi_{xt}\otimes Du+(\phi_{xs}\otimes Dv)DT \\
		  & +(\phi_{yt}\otimes Du)DT+\phi_{tt}Du\otimes Du+(\phi_{ts}Dv\otimes Du)DT+\phi_tD^2u,
	\end{split}
	\end{equation}
where each side is regarded as an $n\times n$ matrix valued at $(x,y)$, $y=T(x)$.

Note that for every $x\in U$, the equality \eqref{e213} holds at point $y=T(x)$, and for any other $y'\in V$ we have
	\[\phi(x,y',u(x),v(y'))\leq0,\]
since $(u,v)\in\mathcal{K}$. Thus, at $(x,T(x))$ we have
	\[\frac{d\phi}{dy}=\phi_y+\phi_sDv=0.\]
By the assumption \eqref{e207}, $\phi_s>0$, we get
	\begin{equation}\label{e302}
		Dv=-\frac{\phi_y}{\phi_s}.
	\end{equation}

Combining \eqref{e301} and \eqref{e302}, we obtain the equation
	\begin{equation}\label{e303}
	\begin{split}
		& \left|\phi_tD^2u+\phi_{tt}Du\otimes Du+2\phi_{xt}\otimes Du+\phi_{xx}\right| \\
	   =& \left|\phi_{xy}+\phi_{xs}\otimes Dv+\phi_{yt}\otimes Du+\phi_{ts}Dv\otimes Du\right|\left|DT\right| \\
	   =& \left|\phi_{xy}-\frac{1}{\phi_s}\phi_{xs}\otimes\phi_y+\phi_{yt}\otimes Du-\frac{\phi_{ts}}{\phi_s}\phi_y\otimes Du\right|\left|DT\right|,
	\end{split}
	\end{equation}
hence by \eqref{e207},
	\begin{equation}\label{e304}
	\begin{split}
		& \left|\det\,\left[D^2u+\frac{\phi_{tt}}{\phi_t}Du\otimes Du+\frac{2}{\phi_t}\phi_{xt}\otimes Du+\frac{1}{\phi_t}\phi_{xx}\right]\right| \\
	   =& \frac{1}{\phi_t^n}\left|\det\,\left[\phi_{xy}-\frac{1}{\phi_s}\phi_{xs}\otimes\phi_y+\phi_{yt}\otimes Du-\frac{\phi_{ts}}{\phi_s}\phi_y\otimes Du\right]\right|\left|\det\,DT\right|.
	\end{split}
	\end{equation}

Equation \eqref{e304} is a second order fully nonlinear PDE of general Monge-Amp\`ere type \cite{GT}.
In the special case of optimal transportation \eqref{e203}--\eqref{e204}, equation \eqref{e304} becomes
	\begin{equation}\label{e305}
		\left|\det\,\left[D^2u-D^2_{xx}c\right]\right|=\left|\det\,D^2_{xy}c\right|\frac{\rho}{\rho^*\circ T}.
	\end{equation}
For the derivation of the optimal transportation equation \eqref{e305}, see \cite{MTW} for more.

Using the notation from the proof of Lemma \ref{l204},
we can now derive the PDE in the near field reflector problem by using the formula \eqref{e304} and constraint function \eqref{e231}. 

Denote $Y_e=Y/|Y|$, $D\rho=(\partial_1\rho,\cdots,\partial_n\rho)$ the gradient of $\rho$, and $D^2\rho=(\partial_i\partial_j\rho)$ the Hessian of $\rho$. By differentiating \eqref{e231}, 
	\begin{eqnarray*}
		&&\phi_t=\frac{1}{\rho},\qquad \phi_{tt}=-\frac{1}{\rho^2},\qquad \phi_{xt}=0,\\
		&&\phi_{x_i}=-\frac{\epsilon\langle e_i,Y_e\rangle}{1-\epsilon\langle X,Y_e\rangle},\qquad \phi_{x_ix_j}=-\frac{\epsilon\langle\partial_je_i,Y_e\rangle}{1-\epsilon\langle X,Y_e\rangle}-\frac{\epsilon^2\langle e_i,Y_e\rangle\langle e_j,Y_e\rangle}{(1-\epsilon\langle X,Y_e\rangle)^2}.
	\end{eqnarray*}
As in \eqref{e247}, at $Y=T(X)$, where $T$ is the optimal mapping in \eqref{e213}, we have
	\begin{equation}\label{e306}
		\frac{\partial_i\rho}{\rho}=\frac{\epsilon\langle e_i,Y_e\rangle}{1-\epsilon\langle X,Y_e\rangle}.
	\end{equation}
Therefore,
	\begin{equation}\label{e307}
		\phi_{x_ix_j}=-\frac{\epsilon\langle\partial_je_i,Y_e\rangle}{1-\epsilon\langle X,Y_e\rangle}-\frac{1}{\rho^2}\partial_i\rho\partial_j\rho,	
	\end{equation}	
and the LHS of equation \eqref{e304} becomes
	\begin{equation}\label{e308}
	\begin{split}
		M(\rho)&:=\left|\det\,\left[D^2\rho+\frac{\phi_{tt}}{\phi_t}D\rho\otimes D\rho+\frac{2}{\phi_t}\phi_{xt}\otimes D\rho+\frac{1}{\phi_t}\phi_{xx}\right]\right| \\
			&=\left|\det\,\left[\partial_i\partial_j\rho-\frac{2}{\rho}\partial_i\rho\partial_j\rho-\rho\frac{\epsilon\langle\partial_je_i,Y_e\rangle}{1-\epsilon\langle X,Y_e\rangle}\right]\right|.
	\end{split}
	\end{equation}

From \eqref{e233},
	\begin{equation}\label{e309}
		\partial_je_i=\left(0,\cdots,0,-\frac{\delta_{ij}}{\sqrt{1-|x|^2}}-\frac{x_ix_j}{(1-|x|^2)\sqrt{1-|x|^2}}\right).
	\end{equation}

In the {special case} $\Omega^*\subset\{y_{n+1}=0\}$,
	\[Y_e=\frac{Y}{|Y|}=(Y_{e,1},\cdots,Y_{e,n},0).\]
So, $\langle \partial_je_i,Y_e\rangle=0$ and \eqref{e308} becomes
	\[M(\rho)=\left|\det\,\left[D^2\rho-\frac{2}{\rho}D\rho\otimes D\rho\right]\right|.\]
Let $u=1/\rho$. We have the standard Monge-Amp\`ere operator as
	\[M(u)=|\det\,D^2u|.\]

In the {general case} when $\Omega^*$ is given by \eqref{e107}, let us now calculate the term $\rho\frac{\epsilon\langle\partial_je_i,Y_e\rangle}{1-\epsilon\langle X,Y_e\rangle}$ in \eqref{e308}. Let $E$ be the supporting ellipsoid of $\Gamma_\rho$ at $X\rho(X)$, with foci $O,Y$ and eccentricity $\epsilon$. Recall that we have the relation \eqref{e250}.
	
Therefore,
	\begin{equation}\label{e310}
	\begin{split}
		\rho\frac{\epsilon\langle\partial_je_i,Y_e\rangle}{1-\epsilon\langle X,Y_e\rangle} &= \frac{\rho\epsilon\langle\partial_je_i,Y\rangle}{|Y|-\epsilon\langle X,Y\rangle} \\
			&=\frac{\rho\langle\partial_je_i,Y\rangle}{\rho+d-\langle X,Y\rangle} \\
	\mbox{from \eqref{e309} }\qquad &=\frac{-\rho}{\rho+d-\langle X,Y\rangle}\left(\frac{y_{n+1}}{x_{n+1}}\right)\left(\delta_{ij}+\frac{x_ix_j}{1-|x|^2}\right).
	\end{split}
	\end{equation}
	
Combining \eqref{e246} into \eqref{e310}, we obtain
	\begin{equation}\label{e311}
		\frac{\rho\epsilon\langle\partial_je_i,Y_e\rangle}{1-\epsilon\langle X,Y_e\rangle}=-\frac{|\nabla\rho|^2+\rho^2}{2\rho d}\left(\frac{y_{n+1}}{x_{n+1}}\right)\mathcal{N}_{ij},
	\end{equation}	
where $\{\mathcal{N}_{ij}\}$ is in \eqref{e111}.
Actually, as one can see from \eqref{e236}, $\mathcal{N}_{ij}=g_{ij}=h_{ij}$ is equal to the metric and the second fundamental form under the \emph{projection coordinates} \eqref{e233}.  
	
Next, let us now calculate the length $d=|Y-X\rho|$ appearing in \eqref{e311}.	
Recall that $\nabla\rho=g^{ij}e_i\partial_j\rho$ satisfies \eqref{e240}--\eqref{e241}. 
Thus, from \eqref{e245}, we have
	\begin{equation}\label{e312}
	\begin{split}
		y_{n+1} &=\frac{d}{|\nabla\rho|^2+\rho^2}\left(-2\rho(D\rho\cdot x)x_{n+1}+(|\nabla\rho|^2-\rho^2)x_{n+1}\right)+\rho x_{n+1} \\
			&=\frac{d}{|\nabla\rho|^2+\rho^2}\left(|D\rho|^2-(\rho+D\rho\cdot x)^2\right)x_{n+1}+\rho x_{n+1}.
	\end{split}
	\end{equation}
Therefore,
	\begin{equation}\label{e313}
		d = \left(\frac{y_{n+1}}{x_{n+1}}-\rho\right)\frac{|\nabla\rho|^2+\rho^2}{|D\rho|^2-(\rho+D\rho\cdot x)^2}. 
	\end{equation}

Finally, combining \eqref{e313} into \eqref{e311} we obtain
	\begin{equation}\label{e314}
		\frac{\rho\epsilon\langle\partial_je_i,Y_e\rangle}{1-\epsilon\langle X,Y_e\rangle}=\frac{|D\rho|^2-(\rho+D\rho\cdot x)^2}{2\rho}\left(\frac{y_{n+1}}{\rho x_{n+1}-y_{n+1}}\right)\mathcal{N}_{ij}
	\end{equation}
Using the notation \eqref{e108}--\eqref{e111}, $a=|D\rho|^2-(\rho+D\rho\cdot x)^2$ and $\mathcal{N}=(\mathcal{N}_{ij})$. We get the LHS of equation \eqref{e304}
	\begin{equation}\label{e315}
	\begin{split}
		M(\rho) &= \left|\det\,\left[D^2\rho-\frac{2}{\rho}D\rho\otimes D\rho-\frac{ay_{n+1}}{2\rho(\rho x_{n+1}-y_{n+1})}\mathcal{N}\right]\right| \\
			&= \left|\det\,\left[D^2\rho-\frac{2}{\rho}D\rho\otimes D\rho-\frac{a(1-t)}{2t\rho}\mathcal{N}\right]\right|.
	\end{split}
	\end{equation}

To compute the RHS of \eqref{e304}, one can directly differentiate the constraint function $\phi$ in \eqref{e231}, but the computations are rather complicated. Instead, we recall a result in \cite{KW} in the following: the Jacobian determinant of the reflection mapping $T_\rho$ is equal to
  \begin{equation}\label{e316}
  \begin{split}
   \left|\det\,DT_\rho\right| &= 2^n\rho^{2n+1}x_{n+1}|\nabla\psi|\left|\frac{t^nb\beta}{a^{n+1}}\right|\left|\det\,\left[D^2\rho-\frac{2}{\rho}D\rho\otimes D\rho-\frac{a(1-t)}{2t\rho}\mathcal{N}\right]\right| \\
			      &= 2^n\rho^{2n+1}x_{n+1}|\nabla\psi|\left|\frac{t^nb\beta}{a^{n+1}}\right|M(\rho),
  \end{split}
  \end{equation}
where $b,\beta$ are defined in \eqref{e109}--\eqref{e110} and $\psi$ is the defining function of $\Omega^*$ in \eqref{e107}. 
Alternatively, one can obtain \eqref{e316} by differetiating the mapping $T_\rho$ in \eqref{e245}. 

On the other hand, from \eqref{e304}
  \begin{equation}\label{e317}
   |\det\,DT|={M(\rho)}{\phi_t^{n}\left|\det\,\left[\phi_{xy}-\frac{1}{\phi_s}\phi_{xs}\otimes\phi_y+\phi_{yt}\otimes Du-\frac{\phi_{ts}}{\phi_s}\phi_y\otimes Du\right]\right|^{-1}}.
  \end{equation}
By Lemma \ref{l204} and Theorem \ref{t101}, $|\det\,DT|=\left|\det\,DT_\rho\right|$. Thus
  \begin{equation}\label{e318}
   \frac{1}{\phi_t^{n}}\left|\det\left[\phi_{xy}\!-\!\frac{\phi_{xs}}{\phi_s}\otimes\phi_y\!+\!\phi_{yt}\otimes Du\!-\!\frac{\phi_{ts}}{\phi_s}\phi_y\otimes Du\right]\right|=\left|\frac{a^{n+1}}{t^nb\beta}\right|\frac{1}{2^n\rho^{2n+1}x_{n+1}|\nabla\psi|}.
  \end{equation}

Note that we projected $\Omega\subset\mathbb{S}^n$ on the $n$ dimensional space $(x_1,\cdots,x_n)$ in \eqref{e233}, $dx=\omega d\mu$, where $d\mu$ is the surface area element of $\Omega$, $\omega=\sqrt{1-|x|^2}$. By Lemma \ref{l204} and \eqref{e226},
  \begin{equation}\label{e319}
   |\det\,DT|=|\det\,DT_\rho|=\frac{f}{\omega g}.
  \end{equation}
Therefore, combining \eqref{e318}--\eqref{e319} into \eqref{e304}, we obtain the equation
  \begin{equation}\label{e320}
   M(\rho)=\left|\frac{a^{n+1}}{t^nb\beta}\right|\frac{f}{2^n\rho^{2n+1}\omega^2g|\nabla\psi|}.
  \end{equation}

This completes the proof of Theorem \ref{t102}. However, note that since we calculate the absolute value for the determinant, the matrix in $M(\rho)$, \eqref{e315} has a different sign to that in \cite{KW}.

\begin{remark}\label{r301}
Another special case of the reflector problem is the far field case \cite{W96}. Suppose a ray $X$ is reflected off by $\Gamma_\rho$ to a direction $Y$. Set the functional and constraint function in \eqref{e230}--\eqref{e231} to be
  \begin{eqnarray}
   I(\rho,\eta) \!\!&=&\!\! \int_\Omega\log\rho(X)f(X)+\int_{\Omega^*}\log\eta(Y)g(Y), \label{e321}\\
   \phi(X,Y,\rho,\eta) \!\!&=&\!\! \log\rho+\log\eta+\log(1-\langle X,Y\rangle). \label{e322}
  \end{eqnarray}
Similarly to Theorem \ref{t101}, one can show that if $(\rho,\eta)$ is a dual maximizing pair of $I$, then $\rho$ is a solution of the far field reflector problem. This formulation was previously obtained by Wang in \cite{W04}. 

The equation in the far field case can be directly obtained by using the formula \eqref{e304} and differentiating the constraint function \eqref{e322}. 
Here we remark that the far field equation is a limit case of \eqref{e315} for the near field one, \cite{KW}. 

To see this, using our notations \eqref{e108}--\eqref{e110}, 
from \eqref{e313} we have the length of reflected ray $d=|Y-X\rho|$ is equal to
  \begin{equation}\label{e323}
    d = \left(\frac{y_{n+1}-\rho x_{n+1}}{x_{n+1}}\right)\frac{|\nabla\rho|^2+\rho^2}{|D\rho|^2-(\rho+D\rho\cdot x)^2} = -t\rho\frac{b}{a}.
  \end{equation}

Let's regard the target $\Omega_r^*=\{rZ\,:\,Z\in\Omega^*_1\}$, where $r$ is sufficiently large, and $\Omega^*_1$ is a domain in the south hemisphere of $\mathbb{S}^n$. 
In this case, the defining function in \eqref{e107} will be $\psi(Z)=r^2-|Z|^2$. 
Let $g_r$ be the light distribution on $\Omega^*_r$ under the same reflector $\Gamma$.
Then when $r$ is sufficiently large, $r^ng_r\to g$, and
  \begin{eqnarray}
   \beta|\nabla\psi|\!\!&=&\!\!\frac{t|\nabla\psi|}{(Y-X\rho)\cdot\nabla\psi}\to\frac{-t}{d}=\frac{a}{\rho b}, \label{e324}\\
   \frac{r}{t}\!\!&=&\!\!\frac{|Y|}{t}\to\frac{d}{t}=-\frac{\rho b}{a}. \label{e325}
  \end{eqnarray}
Sending $r\to\infty$, from \eqref{e112} we obtain the equation for the far field case
  \begin{equation}\label{e326}
   \left|\det\,\left[D^2\rho-\frac{2}{\rho}D\rho\otimes D\rho+\frac{a}{2\rho}\mathcal{N}\right]\right|=\frac{|b|^nf}{2^n\rho^n\omega^2g}.
  \end{equation}
\end{remark}



\begin{thebibliography}{999}

\bibitem{Bre} Brenier, Y.,
		Polar factorization and monotone rearrangement of vector-valued functions,
		\textit{Comm. Pure Appl. Math.} 44 (1991), 375--417.
		
\bibitem{GM} Gangbo, W. and McCann, R. J.,
		Optimal maps in Monge's transport problem,
		\textit{C. R. Acad. Sci. Paris S\'er. I. Math.} 321 (1995), 1653--1658.
		
\bibitem{GT} Gilbarg, D. and Trudinger, N.,
		\textit{Elliptic partial differential equations of second order}.
		Springer-Verlag, Berlin, 1983.
		
\bibitem{GW} Guan, P. and Wang, X.-J.,
		On a Monge-Amp\`ere equation arising in geometric optics,
		\textit{J. Diff. Geom.}, 48 (1998), 205--222.

\bibitem{GH} Guti\'errez, C. E. and Huang, Q.,
		The near field refractor, preprint. 
		
\bibitem{KW} Karakhanyan, A. and Wang, X.-J.,
		On the reflector shape design,
		\textit{J. Diff. Geom.}, 84 (2010), 561--610.
		
\bibitem{KO} Kochengin, S. and Oliker, V.,
		Determination of reflector surfaces from near-field scattering data,
		\textit{Inverse Problems} 13 (1997), 363--373.

\bibitem{Liu} Liu, J.,
		On a class of nonlinear optimization problems, in preparation. 
		
\bibitem{MTW} Ma, X. N.; Trudinger, N. S. and Wang, X.-J.,
		Regularity of potential functions of the optimal transportation problem,
		\textit{Arch. Rat. Mech. Anal.}, 177 (2005), 151--183.
		
\bibitem{Sch} Schneider, R.,
		\textit{Convex Bodies}. The Brunn-Minkowski Theory.
		\textit{Cambridge University Press}, Cambridge, 1993.
		
\bibitem{U} Urbas, J.,
       \textit{Mass transfer problems}, Lecture Notes, Univ. of Bonn, 1998.
       
\bibitem{V} Villani, C.,
		\textit{Optimal transport. Old and new}. Grundlehren Math. Wiss., Vol. 338,
		\textit{Springer-Verlag, Berlin}, 2009.
       
\bibitem{W96} Wang, X.-J.,
		On the design of a reflector antenna,
		\textit{Inverse problems} 12 (1996), 351--375.
		
\bibitem{W04} Wang, X.-J.,
		On the design of a reflector antenna II,
		\textit{Calc. Var. and PDEs} 20 (2004), 329--341.
		

\end{thebibliography}
\end{document}